\documentclass[11pt]{amsart}
\usepackage{amssymb}
\usepackage{dsfont}
\usepackage{color}

\newtheorem{theorem}{Theorem}[section]
\newtheorem{lemma}[theorem]{Lemma}

\newtheorem{corollary}[theorem]{Corollary}

\newtheorem{fact}[theorem]{Fact}
\newtheorem{question}[theorem]{Question}
\newtheorem{remark}[theorem]{Remark}

\theoremstyle{definition}
\newtheorem{df}{Definition}
\newtheorem{example}[df]{Example}

\newcommand{\F}{\mathcal F}
\newcommand{\I}{\mathcal I}

\newcommand{\B}{\mathcal B}
\newcommand{\Pe}{\mathcal P}
\newcommand{\N}{\mathbb N}
\newcommand{\Q}{\mathbb Q}
\newcommand{\R}{\mathbb R}

\newcommand{\ve}{\varepsilon}
\newcommand{\on}{\operatorname}

\author{Marek Balcerzak}
\address{Institute of Mathematics, \L \'od\'z University of Technology,
W\'olcza\'nska 215, 93-005 \L \'od\'z, Poland}
\email {marek.balcerzak@p.lodz.pl}
\author{Micha{\l} Pop\l awski}
\address{Institute of Mathematics, \L \'od\'z University of Technology,
W\'olcza\'nska 215, 93-005 \L \'od\'z, Poland}
\email {michal.poplawski.m@gmail.com}
\author{Artur Wachowicz}
\address{Institute of Mathematics, \L \'od\'z University of Technology,
W\'olcza\'nska 215, 93-005 \L \'od\'z, Poland}
\email {artur.wachowicz@p.lodz.pl}

\title{Ideal convergent subseries in Banach spaces} 
\date{}

\date{}
\begin{document}
\begin{abstract}
Assume that $\I$ is an ideal on $\N$, and $\sum_n x_n$ is a divergent series in a Banach space $X$. 
We study the Baire category, and the measure of the set
$A(\I):=\left\{t \in \{0,1\}^{\N} \colon \sum_n t(n)x_n \textrm{ is } \I\textrm{-convergent}\right\}$.
In the category case, we assume that $\I$ has the Baire property  and $\sum_n x_n$ is not unconditionally convergent, and
we deduce that $A(\I)$ is meager.
We also study the smallness of $A(\I)$ in the measure case when the Haar probability measure $\lambda$ on $\{0,1\}^{\N}$
is considered. If $\I$ is analytic or coanalytic, and $\sum_n x_n$ is $\I$-divergent, then 
$\lambda(A(\I))=0$ which extends the theorem of Dindo\v{s}, \v{S}al\'at and Toma. Generalizing one of their examples, we show
that, for every ideal $\I$ on $\N$, with the property of long intervals, there is a divergent series of reals such that $\lambda(A(\on{Fin}))=0$ and $\lambda(A(\I))=1$. 
\end{abstract}
\maketitle
\section{Introduction}
We consider series in Banach spaces over $\R$. If a given series is divergent, some of its subseries may be convergent.
It is interesting to investigate the set of such subseries from the viewpoints of measure and the Baire category.
A subseries is obtained from a given series $\sum_n x_n$ by the choice of an infinite set of its terms without the change of ordering. So, either we choose an increasing sequence of indices $s(1)<s(2),\dots$ and consider $\sum_n x_{s(n)}$, or we choose a 0-1 sequence $(t(n))$ and consider $\sum t(n)x_n$. One can see that, in the case of usual convergence, these two approaches
are equivalent since the respective sequences of partial sums, of $\sum_n x_{s(n)}$ and $\sum t(n)x_n$, are either convergent or divergent, simultaneously, provided that $t(n)=1$ iff $n\in s[\N]$. The approach dealing with 0-1 sequences can use the structure of $\{0,1\}^\N$ which is a Polish group with the standard probability Haar measure. Hence the coding of subseries by
elements of $\{0,1\}^\N$ is useful to study the size of selected sets of subseries in both measure and category viewpoints
(cf. \cite{RRR}). Increasing sequences $s(1)<s(2),\dots$ of indices form a closed subset $S$ of the Polish space $\N^\N$.
Hence $S$ is also a Polish space, and the Baire category of sets of subseries can be studied when they are coded by elements of $S$ (cf. \cite{K1}).

Our previous paper \cite{BPW} is devoted to studies of the Baire category of sets of subseries that are convergent in a more general sense. Namely, we considered $\I$-convergence where $\I$ is an ideal of subsets of $\N:=\{1,2,\ldots\}$. The only case considered before, going to a similar direction,
was connected with the classic density ideal $\I_d$ (then one speaks about statistical convergence) in the article \cite{DST} where the measure viewpoint was also investigated. 

In \cite{BPW}, we used the coding of subseries by elements of $S$.
Here we decide to use the coding by elements of $\{0,1\}^\N$. Surprisingly, we will show (in Section 2) that, for $\I$-convergence, this approach can be different from that using elements of $S$. In Section 3, we develop the studies of \cite{BPW}. We assume that a given series $\sum_n x_n$ is not unconditionally convergent in a Banach space $X$. 
We establish the meagerness
of set of $\I$-convergent subseries of $\sum_n x_n$ if an ideal $\I$ has the Baire property. Thanks to the coding subseries
by elements of $\{0,1\}^\N$ we can drop additional assumptions that were needed when the approach with $S$-coding was used in \cite{BPW}. We show that the third approach using
parameters from $\{-1,1\}$ is also of some interest because of a known characterization of unconditionally
convergent series.

In Section 4, we detect the measure of the set of $\I$-convergent subseries of a given $\I$-divergent series.
We extend the former results of \cite{DST} dealing with statistical convergence. One of our theorems uses
a new notion of an ideal with the property of long intervals (PLI). We discuss examples connected with
the notions of denseness, PLI, and shift-invariance of ideals. In Section 5, we compare
the category and the measure cases, and we pose some open problems.

Let us recall basic information on ideals and ideal convergence.
By an {\em ideal on} $\N$ we mean an ideal $\I$ of subsets of $\N$ such that $\N\notin\I$ and
$\on{Fin} \subset \I$ where $\on{Fin}$ stands for the family of all finite subsets of $\N$. Since the power set $\Pe(\N)$
can be identified with the Cantor space $\{0,1\}^\N$ (via characteristic functions), an ideal on $\N$ can be treated as a subset of $\{0,1\}^\N$. So, we can speak on $F_\sigma$, Borel, analytic ideals, or ideals with the Baire property, etc.
(cf. \cite{Far}).

The following result (due to Jalali-Naini and Talagrand; see \cite{JN}, \cite{Ta}) gives a useful characterization of
ideals with the Baire property.

\begin{lemma}\label{Tal}
An ideal $\I$ on $\N$ has the Baire property if and only if
there is an infinite sequence $n_1<n_2<\dots$ in $\N$ such that no member of $\I$ contains infinitely many intervals $[n_i,n_{i+1})\cap\N$.
\end{lemma} 

If $(X,\rho)$ is a metric space and $\I$ is ideal on $\N$, we can investigate the notion of $\I$-convergence (see e.g.\cite{KSW}). We say that a sequence $(x_n)$ in $X$ is $\I$-{\em convergent} to $x \in X$ (and write $\I$-$\lim_n x_n=x$) if 
$$\{n \in \N \colon \rho(x_n,x)>\varepsilon \} \in \I$$
 for each $\varepsilon>0$ (clearly, we can use $\rho(x_n,x)\geq\ve$ in this definition). If such a limit exists, then it is unique. Note that $\lim_n x_n=x$ implies $\I$-$\lim_n x_n=x$.
 If we apply $\I$-convergence to the case of the sequence $(\sum_{i=1}^n x_i)_{n\in\N}$ of partial sums of a sequence $(x_n)$ in a normed space, we obtain the notion of an $\I$-convergent series. A series is called $\I$-divergent if it is not $\I$-convergent. The case $\on{Fin}$-convergernce leads us to the classic notion of convergence. Another example is statistical convergence, i.e. $\I_d$-convergence where 
$$\I_d:=\left\{E \subset \N \colon \lim_{n \to \infty} \frac{|E \cap \{1,\ldots,n\}|}{n}=0\right\}$$ 
is called the {\em density ideal}. Statistical convergence of sequences was investigated in many papers (see for instance, 
\cite{S1}, \cite{Mi}, \cite{DMK}, \cite{CGK}). There is only a few articles
on statistical convergence of series  (\cite{DST}, \cite{CST}, \cite{DS}). A general case of ideal (or filter) convergence of series has been also of some interest recently
(see \cite{GO}, \cite{L}, \cite{LO}, \cite{BPW}). Note that a filter is a dual notion to an ideal. If $\I$ is an ideal on
$\N$ then $\F_\I:=\{\N\setminus E\colon E\in\I\}$ is a {\em filter on} $\N$, and if $\F$ is a filter on $\N$ then
$\I_{\F}:=\{\N\setminus E\colon E\in\F\}$ is an ideal on $\N$.

For $E\subset\N$ and $m,k\in\N$, we write $mE\pm k:=\{ mn\pm k\colon n\in E\}$.
An ideal $\I$ on $\N$ is called {\em dense} if each infinite subset of $\N$ has an infinite subset in $\I$.
We say that $\I$ is a P-{\em ideal} if, for any $E_n\in\I$, $n\in\N$, there exists $E\in\I$ with $E_n\setminus E\in\I$ for each $n\in\N$. (See \cite{Far}.)
We say that $\I$ is {\em shift-invariant} if, for every $E\in\I$, we have $E+1\in\I$ (cf. \cite{LO}). One can also consider the {\em two-sided shift-invariance} of $\I$, defined by $(E\pm 1)\cap\N\in\I$ for each $E\in\I$. Clearly, the ideal $\I_d$ is dense and two-sided shift-invariant. It is also a P-ideal.

The following lemma (used in \cite{BPW}) is a consequence of the result by Dems \cite{D}. This is the ideal analogue of the Cauchy condition for series.

\begin{lemma} \label{Cau}
A series $\sum_n x_n$ in a Banach space is $\I$-convergent if and only if
for every $\ve >0$ there exists $m\in\N$ such that the set $\{ j>m\colon\|\sum_{i=m+1}^j x_i\|>\ve\}$ belongs to $\I$. 
\end{lemma}

\section{Subseries -- two approaches}
As it was mentioned in Section 1, the notion of subseries can be considered at least in two manners. It will be important to distinguish
them when one considers $\I$-convergence of subseries (see Example~\ref{E1} below). 
Let us define
$$S:=\{s \in \N^{\N}  \colon \forall{n \in \N} \  s(n)<s(n+1) \} .$$
Then a series of the form $\sum_n x_{s(n)}$, with $s\in S$, will be called an $S$-{\em subseries} of $\sum_n x_n$.
A series of the form $\sum_n t(n)x_n$, with $t\in\{0,1\}^\N$, will be called a 0-1 {\em subseries} of $\sum_n x_n$.
There is a natural homeomorphic embedding $h\colon S\to\{0,1\}^\N$.
Namely, for $s:=(n_k)\in S$, put $h(s):=\chi_{\{n_k\colon k\in\N\}}$.
The image $h[S]$ consists of all sequences having infinitely mamy ones. This is a co-countable subset of $\{0,1\}^\N$
and we can assume that the remaining sequences, with finitely many ones, determine trivially convergent subseries.
In fact, the set $\{0,1\}^\N\setminus h[S]$ is negligible in the sense of measure and the Baire category. 

Let a series $\sum_n x_n$ in a normed space be given, and fix $s\in S$ as above.
It is easy to check that $\sum_n x_{s(n)}$ is convergent if and only if $\sum_n t(n)x_n$ is convergent,
where $t:=h(s)$. However,
in general, the $\I$-convergence of an $S$-subseries
need not be equivalent to the $\I$-convergence of the corresponding 0-1 subseries.
This is shown in the following example.

\begin{example} \label{E1}
Let an ideal $\I$ consist of all sets $E\subset\N$ such that $E\cap(2\N-1)\in\on{Fin}$.
Hence $2\N\in\I$. Then $\I$ is of type $F_\sigma$. Note that $\I$ is neither dense nor shift-invariant.
Define a sequence $(x_n)_{n \in \N}$ of reals as $(1,0,-1,1,0,-1,1,0,-1,\dots)$.
Let $t:=\chi_{\N \setminus (3\N+2)}$ and $s:=h^{-1}(t).$ Then 
$$(x_{s(n)})=(1,-1,1,-1,1,-1,\ldots)\quad\mathrm{and}\quad
(t(n)x_n)=(1,0,-1,1,0,-1,1,0,-1,\ldots).$$
 Hence 
$(\sum_{k=1}^n x_{s(k)})_{n \in \N}=(1,0,1,0,1,0,\ldots)$ is $\I$-convergent to $1$. 
On the other hand, 
$$\left(\sum_{k=1}^n t(k)x_k\right)_{n \in \N}=(1,1,0,1,1,0,1,1,0,\ldots)$$ 
is not $\I$-convergent since its restriction to $2\N-1$ contains infinitely many zeros and ones.

Now, let $u:=\chi_{(6\N-5)\cup (6\N-1)\cup 6\N}$ and $v:=h^{-1}(u)$. Then
$$(x_{v(n)})=(1,0,-1,1,0,-1,1,0,-1,\ldots)\quad\mathrm{and}\quad
(u(n)x_n)=(1,0,0,0,0,-1,1,0,0,0,0,-1,\ldots).$$
 Hence 
 $(\sum_{k=1}^n x_{v(k)})_{n \in \N}=(1,1,0,1,1,0,\ldots)$ is not $\I$-convergent while as  
$$\left(\sum_{k=1}^n u(k)x_k\right)_{n \in \N}=(1,1,1,1,1,0,1,1,1,1,1,0,\ldots)$$
is $\I$-convergent to $1$.
\end{example}

\section{Subseries and the Baire category}
Recall that a series $\sum_n x_n$ in a normed space is  {\em unconditionally convergent} if $\sum_n x_{p(n)}$ is convergent for every permutation $p$ of $\N$.
The following well-known characterization (see \cite[Theorem 1.3.2]{KK}) has inspired us to check the Baire category of some sets of series generated
by a given series which is not unconditionally convergent.

\begin{theorem}\label{uncon}
For a series $\sum_n x_n$ in a Banach space, the following conditions are equivalent:
\begin{itemize}
\item[(i)] the series $\sum_n x_n$ converges unconditionally;
\item[(ii)] all subseries of the given series are convergent;
\item[(iii)] all series of the form $\sum_n t(n) x_n$, with $t\in\{-1,1\}^\N$, are convergent.
\end{itemize}
\end{theorem}

Denote by $P$ the set of all bijections in $\N^\N$. 
Our idea is the following. Let $\sum_n x_n$ be a series which does not converge unconditionally.
Then Theorem \ref{uncon} produces four examples of ``bad'' series which are not convergent: 
$$\sum_n x_{p(n)}\mbox{ with }p\in P,\quad
\sum_n x_{s(n)}\mbox{ with }s\in S,$$ $$ \sum_n t(n)x_n\mbox{ with }t\in\{0,1\}^\N,\quad  \sum_n u(n)x_n\mbox{ with }u\in\{-1,1\}^\N .$$
However, some of series of these four kinds can be convergent. Even possibly more of them are $\I$-convergent
where $\I$ is a regular (in some sense) ideal on $\N$. So, it is interesting to check the Baire category of the sets
of these good series of four kinds.  The respective series are coded by parameters from $P$, $S$, $\{0,1\}^\N$, $\{-1,1\}^\N$
which are Polish spaces with the respective product topologies (in the case of $P$ and $S$ are $G_\delta$ subsets of $\N^\N$, and we use the Alexandrov theorem). The cases of $P$ and $S$ were investigated in \cite{BPW}. It was concluded that the respective sets of ``good'' series are meager (that is, of the first Baire category) under the assumption that $\I$ is a shift-invariant ideal with the Baire property. In the present paper, we will check analogous phenomena for $\{0,1\}^\N$
and $\{-1,1\}^\N$. Basic tools and ideas that we propose are similar to those used in \cite{BPW}. However,
there are some essential differences in comparision with the cases of $S$ and $P$.

Now, we are able to prove a theorem on the meagerness of the set of 0-1 subseries that are $\I$-convergent.
This is an analogous statement to the respective result from \cite{BPW} dealing with $S$-subseries but,
what is interesting, in the 0-1 subseries case, we do not require the shift-invariance of $\I$.

\begin{theorem} \label{zero-one}
Let $\sum_n x_n$ be a series in a Banach space $X$, which is not unconditionally convergent. Let $\I$ be an ideal on $\N$ with the Baire property. Then the set
$$A(\I):=\left\{t \in \{0,1\}^{\N} \colon \sum_n t(n)x_n \textrm{ is } \I\textrm{-convergent}\right\}$$
is  meager  in $\{0,1\}^\N$.
\end{theorem}

\begin{proof}
We follow the method used in the proof of Theorem 3.4 in \cite{BPW}. 
First we use Theorem \ref{uncon} to get a sequence $u \in \{0,1\}^{\N}$ such that $\sum_n u(n)x_{n}$ is divergent. 
Then the Cauchy condition does not hold, so we can find $\varepsilon>0$ such that
\begin{equation} \label{oto}
\forall{m \in \N}  \ \exists{\tau_m>m}  \ \left\|\sum_{i=m+1}^{\tau_m} u(i)x_i \right\|>2\varepsilon.
\end{equation}
Fix a sequence $(n_k)_{k \in \N}$ witnessing that $\I$ has the Baire property according to Lemma \ref{Tal}. 
Let $I_k:=[n_k,n_{k+1}) \cap \N$ for $k\in\N$. Denote in short $A:=A(\I)$.
For any $m \in \N$ and $l>m$ let us define
$$A_{m,l}:=\left\{t \in \{0,1\}^{\N} \colon \exists{k \in \N}\,\left(n_k>l \wedge \forall{j \in I_k}\, \left\|\sum_{i=m+1}^j t(i)x_i \right\|>\varepsilon\right)\right\}.$$
We have $ \bigcap_{m \in \N} \bigcap_{l>m} A_{m,l} \subset \{0,1\}^{\N} \setminus A$. Indeed, if $t \in \bigcap_{m \in \N} \bigcap_{l>m} A_{m,l}$, then for any $m\in\N$ there are infinitely many intervals $I_k$ with $n_k>m$ such that 
$$\left\|\sum_{i=m+1}^j t(i)x_i \right\|>\varepsilon\textrm{ for every }j\in I_k.$$
Then Lemma \ref{Tal} implies that  $\{ j>m\colon\|\sum_{i=m+1}^j x_i\|>\ve\}$ is not in $\I$ for each $m\in\N$
which by Lemma \ref{Cau} proves that $t \notin A$.

It suffices to show that, for any $m \in \N$ and $l>m$,
the set $A_{m,l}$ contains an open dense set. Then $\bigcap_{m \in \N} \bigcap_{l>m} A_{m,l}$ is comeager
which yields the assertion.
 
So, fix $m \in \N$ and $l>m$.
Consider an open set $U$ from the standard base $\B$ of the product topology in $\{0,1\}^\N$
$$U:=\{t\in \{0,1\}^\N\colon t\mbox{ extends }(t_1,\dots, t_r)\} $$
where $t_1,\dots, t_r\in\{0,1\}$. We may assume that $r\geq l$.
We will show that there is a basic open set $V=V_U$ with $V\subset U\cap A_{m,l}$. (Then $\bigcup_{U\in\B} V_U$ is the desired open dense set.)
 We will extend $(t_1,\ldots,t_r)$ to an appriopriate sequence $(t_1,\ldots,t_q)$. Pick the smallest $w>r$ with $u(w)=1$. Put $t_{i}:=0$ for $i=r+1,\ldots ,w-1$, and $t_{w}:=1$.
Consider two cases:

Case 1. $\left\| \sum_{i=m+1}^w t_ix_i \right\|>\varepsilon$. 
Pick the smallest $n_k$ in a sequence $(n_i)$ with $n_k>w$.
Put $q=n_{k+1}-1$ and $t_{i}:=0$ for $i=w+1,\dots ,q$. Let
$$V:=\left\{t \in \{0,1\}^\N \colon t\textrm{ extends } (t_1,\dots , t_q)\right\} .$$
Clearly, $V\subset U$ and for any $t\in V$ and $j\in I_k$, we have
$$\left\| \sum_{i=m+1}^j t(i)x_i \right\| =  \left\|\sum_{i=m+1}^w t_ix_i\right\| >\varepsilon.$$ 
Hence, $V \subset U \cap A_{m,l}$ as desired.

Case 2. $\left\| \sum_{i=m+1}^w t_ix_i \right\| \leq \varepsilon$. 
By (\ref{oto}) pick $\tau_w>w$ with
$\left\|\sum_{i=w+1}^{\tau_w} u(i)x_i \right\|>2\varepsilon$. Put $t_{w+j}=u(w+j)$ for $j=1,\ldots,\tau_w-w$. Then $\left\|\sum_{i=w+1}^{\tau_w}t_ix_i\right\|>2\varepsilon$. Hence
$$\left\|\sum_{i=m+1}^{\tau_w}t_ix_i\right\| \geq -\left\|\sum_{i=m+1}^{w} t_ix_i\right\|+\left\|\sum_{i=w+1}^{\tau_w}t_ix_i\right\| \geq -\varepsilon+2\varepsilon=\varepsilon.$$
Now, we are in the Case 1 where $w$ is replaced by $\tau_w$. 
\end{proof}

The set $A(\I)$ in the above theorem depends in fact on a given series, so if a series is not fixed, we will denote it by $A(\I,(x_n))$.
We can state the following characterization (cf. \cite[Corollary 3.6]{BPW}).

\begin{corollary} \label{Cha}
Let $\I$ be an ideal on $\N$ with the Baire property. A series $\sum_n x_n$ in a Banach space
is unconditionally convergent if and only if $A(\I, (x_n))$ is nonmeager in $\{0,1\}^\N$. 
\end{corollary}
\begin{proof} The necessity is obvious since, $\sum_n x_n$ is unconditionally convergent
if and only if $A(\on{Fin},(x_n))$ equals $\{0,1\}^\N$ (by Theorem \ref{uncon}) which implies that $A(\I, (x_n))=\{0,1\}^\N$,
so  $A(\I, (x_n))$ is clearly nonmeager.
The sufficiency follows directly from Theorem \ref{zero-one}.
\end{proof}

The next theorem concerns $\pm 1$-modifications of series which is not unconditionally convergent.
In fact, we will obtain this result as a consequence of Theorem \ref{zero-one}.

\begin{theorem} \label{1-1}
Let $\sum_n x_n$ be a series in a Banach space $X$, which is not unconditionally convergent. Let $\I$ be an ideal on $\N$ with the Baire property. Then the set
$$B(\I):=\left\{t \in \{-1,1\}^{\N} \colon \sum_n t(n)x_n \textrm{ is } \I\textrm{-convergent}\right\}$$
is  meager  in $\{-1,1\}^\N$.
\end{theorem}

\begin{proof}
By Theorem \ref{uncon}, choose a sequence $u \in \{0,1\}^{\N}$ such that $\sum_n u(n)x_{n}$ is divergent.
Fix a partition $\{M_1,M_2\}$ of $\N$ (one of the sets $M_1, M_2$ may be empty). 
At least one of the series $\sum_{n\in M_1} u(n)x_{n}$, $\sum_{n\in M_2} u(n)x_{n}$ is divergent
(we use the convention $\sum_{n\in M}y_n:=\sum_{n\in\N}y_n\chi_{M}(n)$). Assume that the first series
is divergent -- the second case is analogous. Put $v(n):=-2u(n)\chi_{M_1}(n)$ for $n\in\N$. Then 
$v\in Z:=\{-2,0\}^{M_1}\times \{0,2\}^{M_2}$ and the series $\sum_{n\in\N}v(n)x_n$ is
divergent. By a simple modification of the proof of Theorem \ref{zero-one} we conclude that the set
$$E(\I):=\left\{ z\in Z\colon \sum_n z(n) x_n\textrm{ is } \I\textrm{-convergent}\right\}$$
is  meager  in $Z$.

Assume a nontrivial case when $B(\I)\neq\emptyset$, and fix $w\in B(\I)$. Consider the set $Z$ as above with
$M_1:=w^{-1}[\{-1\}]$ and $M_2:=w^{-1}[\{ 1\}]$. Note that the mapping $\varphi\colon\{-1,1\}^\N\to Z$,
given by $\varphi(t):=t+w$, $t\in\{-1,1\}$, is a homeomorphism. Also, $B(\I)=\varphi^{-1}[E(\I)]$ since the sum (and the difference) of two $\I$-convergent series
is $\I$-convergent, cf. \cite{KMSS}. Since $E(\I)$ is meager in $Z$, it follows that $B(\I)$ is meager in $\{-1,1\}^\N$.
\end{proof}

The above result is a generalization of statements obtained in \cite[Theorem 2.2]{Di1} and \cite[Theorem 2.1]{Di2} and
dealing with $\I:=\on{Fin}$.
We have also the following counterpart of Corollary \ref{Cha} with an analogous proof.

\begin{corollary} \label{ChaCha}
Let $\I$ be an ideal on $\N$ with the Baire property. A series $\sum_n x_n$ in a Banach space
is unconditionally convergent if and only if $B(\I, (x_n))$ is nonmeager in $\{-1,1\}^\N$. 
\end{corollary}

\section{Subseries from the measure viewpoint}
Let $\nu(\{0\})=1/2=\nu(\{1\})$, and consider the product measure $\lambda$ on
$\{0,1\}^\N$ generated by $\nu$ (this is the Haar measure on the Cantor group).
Given an ideal $\I$ on $\N$, and an $\I$-divergent series $\sum_n x_n$
in a Banach space $X$, one can ask whether the set $A(\I)$ considered
in Theorem \ref{zero-one} is measurable. The value $\lambda(A(\I))$ would be also interesting.
The answer in the cases when $\I:=\on{Fin}$, and $\I:=\I_d, X:=\R$ were given in \cite{DMS} and \cite{DST}, respectively,
where it was shown that $\lambda(A(\I))=0$. The approach in \cite{DMS} and \cite{DST} is different but
equivalent to that of ours since the authors of \cite {DMS} and \cite{DST} use the Lebesgue measure on $[0,1]$, 
and $\{0,1\}^\N$ is 
transformed into $[0,1]$ via binary expansions.

We extend the above results to a more general setting, using similar arguments
and auxiliary facts taken from \cite{BGW}.
We say that $E\subset \{0,1\}^\N$ is a {\em tail set} if, whenever $x\in E$ and $y\in\{0,1\}^N$ differs from $x$ in a finite number of coordinates, then $y\in E$. By the 0-1 law (cf. \cite[Theorem 21.3]{Ox}), if a tail set is measurable, its measure is either $0$ or $1$.

\begin{theorem} \label{meas}
Let $\I$ be an ideal on $\N$, which is analytic or coanalytic. Then for any series $\sum_n x_n$ in a Banach space, 
$\lambda(A(\I))$ is either 0 or 1. If $\sum_n x_n$ is $\I$-divergent, then $\lambda(A(\I))=0$. 
\end{theorem} 
\begin{proof}
Let $\Q_+$ stand for the set of positive rationals. We write in short $A:=A(\I)$. By Lemma \ref{Cau}, we have
$A=\bigcap_{\ve\in\Q_+}\bigcup_{m\in\N} A_{\ve m}$ where
$$A_{\ve m}:=\left\{ t\in\{0,1\}^\N\colon\left\{ j>m\colon\left\|\sum_{i=m+1}^j t(i)x_i\right\|>\ve\right\}\in\I\right\} .$$
As in \cite[Proposition 3.1]{BGW}, we observe that the function $f_{\ve m}\colon\{0,1\}^\N\to\{0,1\}^\N$, given by 
$f_{\ve m}:=\chi_{\left\{ j>m\colon\left\|\sum_{i=m+1}^j t(i)x_i\right\|>\ve\right\}}$,
is continuous. Since $A_{\ve m}=f^{-1}_{\ve m}[\I]$, and $\I$ is analytic or coanalytic, so is $A_{\ve m}$,
for any $\ve$ and $m$. Consequently, $A$ is measurable.
Then observe that $A$ is a tail set, hence either $\lambda(A)=0$ or $\lambda(A)=1$.
Let $\sum_n x_n$ be $\I$-divergent and suppose that $\lambda(A)=1$. 
Let $\varphi(t):=\pmb 1-t$ for $t\in\{0,1\}^\N$ where $\pmb 1:=(1,1,\dots)$. 
Since $\lambda$ is a Haar measure, we have $\lambda (\varphi[A])=1$.
Pick $s\in A\cap\varphi[A]$ and let $s=\varphi(t)$ where $t\in A$. 
Since $s,t\in A$, the series $\sum_n s(n)x_n$ and $\sum_n t(n)x_n$ are $\I$-convergent.
This implies that their sum $\sum_n x_n$ is $\I$-convergent (cf. \cite[Theorem 2.1]{KMSS}), a contradiction.
\end{proof}

From \cite{DST} it follows that, for $\I:=\I_d$, the second assertion of 
Theorem \ref{meas} does not hold if we replace $\I$-divergence of $\sum_n x_n$ by its divergence.
Namely, in \cite[Example 2.4]{DST}, the authors show a divergent series with real terms such that 
$\lambda(A({\on{Fin}}))=0$ while as
$\lambda(A(\I_d))=1$. In Theorem \ref{Slovak}, we will use a similar method to obtain the same effect for a wide class of  ideals on $\N$ that contains $\I_d$. First, let us describe this class.

We say that an ideal $\I$ on $\N$ has the {\em property of long intervals} (in short, PLI)
if,  there exists a sequence $(m(n))\in\N^\N$ such that $\bigcup_{n\in\N}\{m(n),m(n)+1,\ldots,m(n)+n-1\} \in\I$. 
An equivalent version of this property states that there exist a set $E\in\I$ and a sequence $(m(n))\in\N^\N$ such that 
$E\supset \bigcup_{n\in\N}\{m(n),m(n)+1,\ldots,m(n)+n-1\}$.
We then say that $E$ is a {\em witness} of PLI for $\I$.
Note that an ideal $\I$ has the property of long intervals if and only if the dual filter has the unbounded gap property \cite{L}.
We omit an easy proof of another equivalent formulation of this property.

\begin{fact}  \label{FFact}
An ideal $\I$ on $\N$ has the property of long intervals if and only if,
for every increasing sequence $(n_k)$ of positive integers,  there is an increasing sequence $(m_k)$ of positive integers such that 
$$\bigcup_{k\in\N}\{m_k,m_k+1,\ldots,m_k+n_k-1\} \in\I\;\textrm{ and }\; m_k+n_k-1<m_{k+1}\;\textrm{ for all }k\in\N .$$
\end{fact}

\begin{theorem} \label{Slovak}
Assume that $\I$ is an ideal on $\N$, with the property of long intervals.  Then there exists a divergent
series $\sum_n x_n$ in $\R$, with $x_n\not\to 0$, for which $\lambda(A(\I))=1$. Consequently, $\lambda(A(\on{Fin}))=0<1=\lambda(A(\I))$.
\end{theorem}
\begin{proof}
Define an increasing sequence $(m_k)$ of integers as follows.  Let $n_k:=2^{k+1}$ for $k\in\N$, and apply Fact \ref{FFact} to the sequence $(n_k)$. We obtain an increasing sequence $(m_k)$ such that $m_k+2^{k+1}-1<m_{k+1}$ for all $k\in\N$, and
$$E:=\bigcup_{k \in\N} \{m_k,m_k+1,\ldots,m_k+2^{k+1}-1\}\in\I . $$
 
We define terms $x_i$ of a series for indices $i$ belonging to consecutive blocks 
of integers $\{ m_k, m_k+1,\dots ,m_{k+1}-1\}$ for $k\in\N$. Namely, fix $k\in\N$ and let 
\begin{displaymath}
x_{m_k+j}:=\left\{ \begin{array}{ll}
2^{-k} & \textrm{for $j=0,1,\dots ,2^k-1$}\\
-2^{-k} & \textrm{for $j=2^k,2^k+1,\dots ,2^{k+1}-1$}\\
0 & \textrm{for $j=2^{k+1},2^{k+1}+1,\dots , m_{k+1}-m_k-1.$}
\end{array} \right.
\end{displaymath}
Let additionally, $x_n:=0$ for $n<m_1$.
Clearly, $\sum_n x_n$ is divergent since each of the values 0 and 1 is achieved infinitely many times by partial sums of this series. 
Hence $\lambda(A(\on{Fin}))=0$ by Theorem \ref{meas}.

We will prove that $\lambda(A(\I))=1$. We follow the probability methods used in \cite[Example 2.4]{DST}.
For each $k\in\N$, define a random variable
$$X_k(t):=\sum_{n=m_k}^{m_k+2^{k+1}-1}t(n) x_n,\quad t\in\{0,1\}^\N.$$
Then $\mathbb EX_k=0$ and $\on{Var} X_k=1/2^k$. By the Chebyshev inequality, we have
$\lambda(B_k)\leq k^4/2^k$ where $B_k:=\{t\in\{0,1\}^\N\colon |X_k(t)|\geq 1/k^2\}$.
Hence the series $\sum_k\lambda(B_k)$ is convergent. So, $\lambda(\bigcap_{l\in\N}\bigcup_{j\geq l}B_j)=0$
by the Borel-Cantelli lemma. Let $B:=\bigcup_{l\in\N}\bigcap_{j\geq l}B_j^c$. Then $\lambda(B)=1$. 

Finally, we will show that $B\subset A(\I)$ which yields the assertion.
Let $t\in B$ and pick $l\in\N$ such that $t\in\bigcap_{j\geq l}B_k^c$.
Denote $D_n(t):=\sum_{i=1}^n t(i) x_i$ for $n\in\N$, and
put
 $$F:=\bigcup_{k\in\N}\{n\in\N\colon m_k+2^{k+1}\leq n \leq m_{k+1}-1\} .$$
 Then $F$ is in the dual filter of $\I$, by the choice of $E$ . So,
 it is enough to prove that
 the sequence $(D_n(t))_{n\in F}$ is convergent. (Then $\sum_n x_n$ is $\I$-convergent
 and so, $t\in A(\I)$.) 
 
 From $t\in\bigcap_{j\geq l}B_j^c$ it follows 
that  $|X_j(t)|\leq 1/j^2$. Thus for any $p>k\geq l$,
 \begin{equation} \label{Cauchy} 
 |D_{m_p-1}(t)-D_{m_k-1}(t)|\leq \sum_{j=k+1}^\infty|D_{m_{j+1}-1}(t)-D_{m_j-1}(t)|
 =\sum_{j=k+1}^\infty|X_j(t)|\leq\sum_{j=k+1}^\infty\frac{1}{j^2}.
 \end{equation}
 Hence the sequence $(D_{m_j-1}(t))_{j\in\N}$ is Cauchy since $\sum_{j=k+1}^\infty1/j^2\to 0$ if $k\to\infty$.
 By the definition of $(x_n)$, for each $j\in\N$ we have
 $$D_{m_j+2^{j+1}-1}(t)=D_{m_j+2^{j+1}}(t)=\dots =D_{m_{j+1}-1}(t).$$
 So, by (\ref{Cauchy}), the sequence $(D_n(t))_{n\in F}$ is Cauchy. Hence it is convergent as desired.
\end{proof}

We will show that ideals from a wide class have PLI. Let us introduce a weaker version of this property.
We say that an ideal $\I$ on $\N$ has the {\em weak property of long intervals} (in short, wPLI)
if,  for every $n\in\N$, there exists a sequence $(m^n_j)_{j\in\N}\in\N^\N$ such that 
$m^n_j+n-1<m^n_{j+1}$ for each $j\in\N$ and $\bigcup_{j\in\N}\{m^n_j,m^n_j+1,\ldots,m^n_j+n-1\} \in\I$. 
An equivalent version of this property states that, for every $n\in\N$, there exist a set $E_n\in\I$ and a sequence $(m^n_j)_{j\in\N}\in\N^\N$ such that $m^n_j+n-1<m^n_{j+1}$ for each $j\in\N$, and $E_n\supset \bigcup_{j\in\N}\{m^n_j,m^n_j+1,\ldots,m^n_j+n-1\}$.
We then say that the sets $E_n$, $n\in\N$, are {\em witnesses} of wPLI for $\I$.

\begin{remark} \label{rem}
\begin{itemize}
\item[(i)] If $\I\neq\on{Fin}$ is shift-invariant, then it has wPLI. Indeed, pick an infinite set $E\in\I$. Then
$\bigcup_{j=0}^{n-1}(E+j)$, $n\in\N$, are witnesses of wPLI for $\I$.
\item[(ii)] If $\I$ is dense, then it has wPLI. Indeed, pick an infinite set $B_0\in\I$ and choose, inductively,
infinite sets $B_j\subset B_{j-1}+1$ in $\I$ for $j\in\N$. Then $\bigcup_{j=0}^{n-1}B_j$, $n\in\N$, 
are witnesses of wPLI for $\I$.
\item[(iii)] If $\I$ is a P-ideal with wPLI, then it has PLI. Indeed, let $E_n$, $n\in\N$ be witnesses of wPLI for $\I$.
Pick $E\in\I$ such that $E_n\setminus E\in\on{Fin}$ for each $n\in\N$. Then $E$ is a witness of PLI for $\I$.
\end{itemize}
\end{remark}
By (ii) and (iii), every dense P-ideal has PLI.
Several classes of analytic P-ideals are described in \cite{Far}, among them, summable ideals, Erd\H{o}s-Ulam ideals;
many of them are dense. Also a class of density-type P-ideals was studied in \cite{BDFS}.
So, there are many known classes of ideals that have PLI, and then our Theorem \ref{Slovak} is applicable.

Our next result will show that, for some type of non-dense ideals on $\N$, the conclusion of the previous theorem cannot hold.

\begin{theorem} \label{Michael}
Let $\I\neq\on{Fin}$ be an ideal on $\N$ for which there exists an infinite set $E\subset\N$ such that:
\begin{itemize}
\item[(I)] no infinite subset of $E$ belongs to $\I$; 
\item[(II)] $M:=\sup\{e_{n+1}-e_n\colon n\in\N\}<\infty$ where $(e_n)_{n\in\N}$ is an increasing enumeration of $E$.
\end{itemize}
Then, for each normed space $X$ and for every divergent series $\sum_n x_n$ in $X$ which is $\I$-convergent, 
we have $\lambda(A(\I))=0$.
\end{theorem}
\begin{proof}
Assume that a series $\sum_n x_n$ in a normed space $X$ is divergent and the sequence of its partial sums is $\I$-convergent to $x\in X$. 

\noindent {\bf Claim.} If $(r_n)$ is a subsequence of $(e_n)$, then $\sum_{i=1}^{r_n} x_i\to x$ if $n\to\infty$.

Indeed, given $\ve>0$, we have
$$\left\{ r_n\colon n\in\N\textrm{ and } \left\| \sum_{i=1}^{r_n} x_i-x\right\|>\ve\right\}
\subset\left\{ m\in\N\colon \left\| \sum_{i=1}^m x_i-x\right\|>\ve\right\}\in\I .
$$
So, by (I) we obtain $\left\{ r_n\colon n\in\N\textrm{ and } \left\| \sum_{i=1}^{r_n} x_i-x\right\|>\ve\right\}\in\I\cap\Pe(E)\subset\on{Fin}$, which shows the Claim.

Since $\I\neq\on{Fin}$ and (I) holds, the set $C:=\N\setminus E$ is infinite. Note that $\N$ is a union of consecutive finite intervals contained either in $E$ or in $C$. From (II) it follows that the length of each interval contained in $C$ does not exceed $M$.
Pick two subsequences $(b_n)$ and $(d_n)$ of $(e_n)$ consisting of the respective (left and right) end-points of consecutive intervals of $E$. So, we have $b_n\leq d_n<b_{n+1}$, $[b_n,d_n]\cap\N\subset E$, $b_n-1\notin E$, and $d_n+1\notin E$ for each $n\in\N$. 

Since the series $\sum_n x_n$ is divergent, its partial sums cannot converge to $x$. So, pick $\ve>0$ and an increasing sequence $(c_k)$ of integers in $C$ such that $\|\sum_{i=1}^{c_k}x_i-x\|>2\ve$. 
We have $c_k\in (d_{n_k},b_{n_k+1})$ for the respectively chosen indices $n_k$. 
By Claim, $\|\sum_{i=1}^{d_{n_k}}x_i -x\|<\ve$ for all but finitely many $k's$. For simplicity, assume that it hold for all $k's$. Then for each $k\in\N$,
\begin{equation} \label{Noti}
\left\|\sum_{i=d_{n_k}+1}^{c_k}x_i\right\|\ge\left\|\sum_{i=1}^{c_k}x_i-x\right\| -\left\|\sum_{i=1}^{d_{n_k}}x_i-x\right\|>2\ve-\ve=\ve .
\end{equation}

Now, observe that $A(\I)$ is disjoint from $\bigcap_{j\in\N}\bigcup_{k>j}Z_k=\emptyset$ where
$$Z_k:=\left\{t\in\{0,1\}^\N\colon t(d_{n_k}+1)=\ldots=t(c_{k})=1 \textrm{ and }t(c_{k}+1)=\ldots=t(b_{n_k+1})=0\right\} .$$
Indeed, suppose that $t\in A(\I)\cap\bigcap_{j\in\N}\bigcup_{k>j}Z_k$.
Since $t\in A(\I)$, the partial sums of $\sum_n t(n)x_n$ are $\I$-convergent to $y\in X$.
Since $t\in\bigcap_{j\in\N}\bigcup_{k>j}Z_k$, the condition defining $Z_k$ holds for infinitely many $k's$. We should pass to a subsequence but, for simplicity, let us assume that the condition holds for each $k\in\N$. On the one hand, by (\ref{Noti}), we have 
$$\left\|\sum_{i=1}^{b_{n_k+1}} t(i)x_i-\sum_{i=1}^{d_{n_k}} t(i)x_i\right\|=\left\|\sum_{i=d_{n_k}+1}^{c_k} x_i\right\|>\ve\;\textrm{ for all}\;k\in\N .$$
On the other hand,
applying the Claim to $\sum_n x_n$ replaced by $\sum_n t(n)x_n$, we have
$$\sum_{i=1}^{b_{n_k+1}} t(i)x_i-\sum_{i=1}^{d_{n_k}} t(i)x_i\to y-y=0\;\textrm{ if }\;k\to\infty ,$$
a contradiction.

Finally, it suffices to show that $\lambda(\bigcup_{k>j}Z_k)=1$ for each $k\in\N$.
Note that the events $Z_k$, $k\in\N$, are independent. Considering complements, we obtain
$$\lambda\left(\bigcap_{j<k\leq m}(Z_k)^c\right)=\prod_{j<k\leq m}\left(1-\frac{1}{2^{b_{n_k+1}-d_{n_k}}}\right)\leq
\left(1-\frac{1}{2^M}\right)^{m-j}\to 0\;\textrm{ if }\;m\to\infty .$$
\end{proof}

From Theorem \ref{Michael} it follows that the implication in the second assertion of Theorem \ref{meas} cannot be reversed.
Indeed, fix an infinite co-infinite set $E\subset\N$ satisfying condition (II). Consider, for instance, the ideal
$\I$ equal to $\{ B\subset\N\colon B\cap E\in\on{Fin}\}$. Then $\I$ is $F_\sigma$ and, by Theorem \ref{Michael}, in a Banach space $X$, we have $\lambda(A(\I))=0$ for every divergent series in $X$ that is $\I$-convergent.

We finish this section with two examples of ideals that show some further interplays between the notions of denseness, PLI and shift-invariance (cf. Remark \ref{rem}, (i)--(ii)). 

\begin{example} \label{Id1}
Let $\I$ consist of all sets $E\subset\N$ such that $E\subset\{ 2^n\pm j\in\N\colon n\in\N,\; j\in\{ 0,\dots ,k\}\}$ 
for some $k\in\N$. Observe that $\I$ is a two-sided shift-invariant ideal on $\N$. Note that $\I$ is not dense. Indeed, let $B$ consist of all centers $c_n$ of intervals $[2^n, 2^{n+1}]$, $n\in\N$. If we consider an infinite subset $\{ c_{n_i}\colon i\in\N\}$ of $B$ and suppose that it is covered by $\{ 2^n\pm j\in\N\colon n\in\N,\; j\in\{ 0,\dots ,k\}\}$ 
for some $k\in\N$, we obtain a contradiction since we can pick $i\in\N$ large enough to have the distance between $c_{n_i}$ and $2^{n_i}$ (the same as that between $c_{n_i}$ and $2^{n_i+1}$) greater than $k$.
\end{example}

\begin{example} \label {Id2}
Let $\I$ be the smallest ideal containing the family of all sets of the form
$\{ b_n \colon n\in\N\}$ where $(b_n)$ is an increasing sequence of positive integers such that
$b_{n+1}-b_n\to\infty$. Denote by $\B$ the set of all such sequences.
Note that $\I$ is dense since every one-to-one sequence of positive integers has
a subsequence in $\B$. We will show that $\I$ does not have PLI. 
Note that this ideal was studied in \cite{KN} and denoted by ${\mathcal Lac}$.
Suppose that $\I$ has PLI. Let $A:=\bigcup_{j=1}^k\{b_n^j\colon n\in\N\}$ be a witness of PLI for $\I$ where $k\in\N$ and
$(b_n^j)_{n\in\N}\in\B$ for $j\leq k$. Pick $n_0\in\N$ such that for all $n>n_0$ and $j\leq k$ we have
$b_{n+1}^j-b_n^j>k$. Since $A$ is a witness of PLI for $\I$, pick $a>\max\{b^j_{n_0}\colon 1\leq j\leq k\}$ such that $J:=\{a,a+1, \dots, a+k\}\subset A$. By the choice of $a$, each set $J\cap\{b^j_n\colon n\in\N\}$, $1\leq j\leq k$, has
at most one element. So, the cardinality of $J=\bigcup_{j=1}^k(J\cap\{b^j_n\colon n\in\N\})$ is not greater than $k$.
This is a contradiction since $J$ has $k+1$ elements.
\end{example}

\begin{question} \label{que1}
Does the statement of Theorem \ref{Slovak} remains true if we replace PLI by wPLI for an ideal $\I$?
In particular, does it work for the ideal $\I$ from Example \ref{Id2}?
\end{question}

\section{Concluding remarks}
Corollary \ref{Cha} yields a characterization of series $\sum_n x_n$ in a Banach space for which
$A(\I, (x_n))$ is meager (or, nonmeager) provided that $\I$ is an ideal with the Baire property.
It would be interesting to find an analogous characterization in the measure case.
This problem seems difficult but let us shed some light on it. 
By Theorem \ref{meas},
if an ideal $\I$ is analytic or coanalytic, then the measure of $A(\I, (x_n))$ is 
either 0 or 1. However, it is rather hard to describe those series for which any of these cases holds, for a fixed class of ``good'' ideals, or for a fixed nontrivial ideal such as
$\I_d$.

Look at the case of usual convergence. For any series $\sum_n x_n$ in $\R$, the following nice equivalence is known:
\begin{equation}\label{iff}
\lambda(A(\on{Fin},(x_n))=1 \iff (\textrm{the series }\sum_n x_n,\;\sum_n x_n^2\textrm{ are convergent).}
\end{equation}
Indeed, if $\lambda(A(\on{Fin},(x_n))=1$ then $\sum_nx_n$ is not divergent by Theorem \ref{meas}.
Hence $\sum_n x_n$ is convergent, and in this case, $\lambda(A(\on{Fin},(x_n))=1$ iff $\sum_{n=1}^\infty x_n^2<\infty$ which is proved in \cite[Theorem 3]{RRR} and \cite[Theorem 2.2]{DST}.
The equivalence (\ref{iff}) can easily be extended to series in $\R^k$ or in a finite-dimensional Banach space, with 
$\sum_n x_n^2$
replaced by $\sum_n \|x_n\|^2$. 
Note that (\ref{iff}) need not hold 
in the infinite-dimensional case. Indeed, in the Banach space $c_0$, consider the sequence $x_n:=(0,\dots,0,{1}/{\sqrt{n}},0,\dots)$, $n\in\N$, where ${1}/{\sqrt{n}}$ appears on the $n$th place. Then $\sum_{n=1}^{\infty} ||x_n||^2=\sum_{n=1}^\infty {1}/{n}=\infty$ while as $A(\on{Fin},(x_n))=\{0,1\}^\N$. So, we can ask about the respective characterization of $\lambda(A(\on{Fin},(x_n))=1$
for a series $\sum_n x_n$ in infinite-dimensional Banach spaces.

Returning to our general problem dealing with the measure of $\I$-convergent subseries in $\R$, one may expect that a characterization of $\lambda(A(\I,(x_n))=1$,
for a ``good'' ideal $\I$, is analogous to that in (\ref{iff}) where the convergence of $\sum_n x_n$ (on the right-hand side) is replaced by its
$\I$-convergence. However, this cannot be true for
an ideal $\I$ that has PLI. Indeed, for the $\I$-convergent series considered in Theorem \ref{Slovak}, we have $\lambda(A(\I,(x_n))=1$,
but $\sum_{n=1}^\infty x_n^2=\infty$ since $x_n\not\to 0$.

Recall  that, for usual convergence of subseries in $\R$, there is an asymmetry between the category and the measure cases; cf.
\cite[p. 195]{RRR}. The category counterpart of (\ref{iff}) is the following: for every series
$\sum_n x_n$ in $\R$, the set $A(\on{Fin}, (x_n))$ is comeager (nonmeager) iff $\sum_{n=1}^\infty |x_n|<\infty$. 
(See \cite[Theorem 3]{RRR}; this fact is generalized in Corollary \ref{Cha}.) Hence, almost every subseries of 
$\sum_n (-1)^n/n$ is convergent in the measure sense, whereas almost every subseries of $\sum_n (-1)^n/n$ is divergent
in the category sense. Hence, the Cantor space $\{0,1\}^\N$ can be partitioned into the Borel meager set
$A(\on{Fin},(x_n))$ and its complement which is of measure zero. This is another example of parition of a Polish space into two parts that are small in a different sense (see \cite{Ox}, for other natural partitions).

\begin{remark} {\em Let $\I$ be an ideal on $\N$ with the Baire property. For every series $\sum_n x_n$ in a Banach space,
if $A(\I, (x_n))$ is of measure zero, then $A(\I, (x_n))$ is meager. Indeed, suppose  that $A(\I, (x_n))$ is nonmeager. Then 
$\sum_n x_n$ is unconditionally convergent by Theorem \ref{zero-one}. Hence $(A(\I, (x_n)))$ is of measure 1 since
it equals $\{0,1\}^\N$ by Theorem \ref{uncon}.}
\end{remark}

In this paper, we have not touched natural problems connected with the measure of the set of $\I$-convergent $\pm 1$-modifications of a given series
(again one can consider the standard product measure on $\{-1,1\}^\N$).
This is another interesting topic of research. Note that some results in this direction, for usual convergence, were obtained by Dindos in \cite{Di1} for series in a Hilbert space.

\noindent{\bf Acknowledgement.} We would like to thank Andrzej Komisarski for his useful remarks on the proof of Theorem \ref{Slovak}.

\end{document}